\documentclass[12pt]{article}
\usepackage[english]{babel}

\usepackage{upref,amsfonts,amsxtra,a4wide,latexsym,enumerate}
\usepackage{amssymb,a4wide,epsf,mathrsfs}
\usepackage{amsthm}
\usepackage{amsmath}
\usepackage{lastpage}
\usepackage[all]{xy}
\usepackage{cleveref}
\usepackage{pdfpages}
\usepackage{enumitem}
\usepackage{graphicx}
\usepackage{float}

\usepackage{tikz}

\newcommand{\Z}{{\mathbb Z}}

\renewcommand{\phi}{{\varphi}}

\theoremstyle{plain}
\newtheorem{thm}{Theorem}
\newtheorem{prop}{Proposition}
\newtheorem{lemma}{Lemma}
\newtheorem{cor}{Corollary}
\theoremstyle{definition}
\newtheorem{defi}{Definition}
\newtheorem{obs}{Observation}

\newtheorem{claim}{Claim}
\newtheorem{case}{Case}

\crefname{defi}{Definition}{Definitions}
\Crefname{defi}{Definition}{Definitions}
\crefname{thm}{Theorem}{Theorems}
\Crefname{thm}{Theorem}{Theorems}
\crefname{lemma}{Lemma}{Lemmas}
\Crefname{lemma}{Lemma}{Lemmas}
\crefname{cor}{Corollary}{Corollaries}
\Crefname{cor}{Corollary}{Corollaries}
\crefname{prop}{Proposition}{Propositions}
\Crefname{prop}{Proposition}{Propositions}
\crefname{prob}{Problem}{Problems}
\Crefname{prob}{Problem}{Problems}
\crefname{conj}{Conjecture}{Conjectures}
\Crefname{conj}{Conjecture}{Conjectures}

\newtheorem{innercustomthm}{Theorem}
\newenvironment{customthm}[1]
  {\renewcommand\theinnercustomthm{#1}\innercustomthm}
  {\endinnercustomthm}

\begin{document}

\title{Exponentially many $\Z_5$-colorings in simple planar graphs}


\author{Rikke Langhede \qquad  Carsten Thomassen\thanks{Supported by Independent Research Fund Denmark, 8021-002498 AlgoGraph.}\\
\small Department of Applied Mathematics and Computer Science  \\[-0.8ex]
\small Technical University of Denmark  \\[-0.8ex]
\small DK-2800 Lyngby, Denmark  \\
\small\tt rimla@dtu.dk \qquad ctho@dtu.dk}

\maketitle

\begin{abstract}
Every planar simple graph with $n$ vertices has at least $2^{n/9}$ $\Z_5$-colorings. 
\end{abstract}

\section{Introduction}

List coloring and group coloring are generalizations of (ordinary) graph coloring.
While the two generalizations are formally unrelated, it is believed that group coloring is more difficult than list coloring. 
Specifically, it is conjectured in \cite{kpv} that, for every graph $G$, the list chromatic number is less than or equal to the group chromatic number which is defined as the smallest $k$ such that $G$ is $\Gamma$-colorable (defined below) for every group $\Gamma$ of order at least $k$. 
In \cite{lt} it is conjectured that the list chromatic number is even less than or equal to the weak group chromatic number which is the smallest $k$ such that $G$ is $\Gamma$-colorable for \textbf{some} group $\Gamma$ of order $k$. 
(In \cite{lt} it is proved that the two group chromatic numbers are bounded by each other, but may differ by a factor close to 2.)

Group connectivity and group coloring are introduced by Jaeger et al. in \cite{jlpt}. 
For planar graphs they are dual concepts. 
It was shown in \cite{dlms} that graphs with an edge-connectivity condition imposed have exponentially many group flows for groups of order at least 8. 
\cite{tweak} proved the weak 3-flow conjecture, specifically, every 8-edge-connected graph has a nowhere zero 3-flow. 
In \cite{ltwz} the proof was refined to 6-edge-connected graphs, and in \cite{dms} that refinement was used to prove that every 8-edge-connected graph has exponentially many nowhere-zero 3-flows. 

The groups of order 3, 4, 5 are particularly interesting because they relate to the 4-color theorem and Tutte's flow conjectures.  
Jaeger et al. \cite{jlpt} conjectured that every 3-edge-connected graph is $\Z_5$-connected, which is a strengthening of Tutte's 5-Flow Conjecture.
In \cite{t2} it was proven that planar graphs have exponentially many 5-list-colorings (for every list assignment to the vertices), and in \cite{t5} it was proven that planar graphs of girth at least 5 have exponentially many 3-list-colorings (for every list assignment to the vertices). 
Perhaps somewhat surprising, the list-color proof in \cite{t1} carries over, word for word, to a proof of Theorem \ref{thm:A} below.

\begin{customthm}{A} \label{thm:A}
Let $G$ be a simple planar graph. Then $G$ is $\Z_5$-colorable.
\end{customthm} 

Also, the proof in \cite{t5} needs only minor modifications to give the analogous result for group coloring (saying that planar graphs of girth at least 5 have exponentially many $\Z_3$-colorings). 
However, the proof in \cite{t2} does not immediately extend to group coloring.
In this paper we prove

\begin{customthm}{B} \label{thm:B}
Every planar simple graph with $n$ vertices has at least $2^{n/9}$ $\Z_5$-colorings. 
\end{customthm}

Note, that Theorem \ref{thm:A} proves the conjecture of Jaeger et al. when restricted to planar graphs. Theorem \ref{thm:B} even shows that there are many solutions for this family of graphs.

Although both the bound in Theorem \ref{thm:B} as well as the strategy of proof are identical to those in \cite{t2}, the details are significantly different. 
The main idea in \cite{t2} (as well as in the present paper) is an application of the 5-list-color theorem in \cite{t1}. 
In \cite{t1} two neighboring vertices on the outer cycle are allowed to be precolored. 
In \cite{t2} (and in the present paper) we need the extension where a path with three vertices on the outer cycle is precolored. 
Such a coloring cannot always be extended, but the exceptions (called \textit{generalized wheels} in \cite{t2}) are easily characterized and studied. 
Their nice behaviour allows exponentially many list colorings. 
For group colorings, however, there are more exceptions (which we call \textit{generalized multi-wheels}) and, more important, their group coloring properties are far more subtle. 
Here, the group structure is essential, and the proof does not extend to e.g. DP-colorings \cite{dp}. We do conjecture, though, that the planar graphs have exponentially many DP-colorings. That would be a common generalization of \cite{t2} and the present paper. 

Also, we conjecture that if every graph in a graph family has exponentially many $\Z_k$-colorings, then it has exponentially many $k$-list-colorings.
As mentioned earlier, a graph may be $\Gamma$-colorable and non-$\Gamma'$-colorable for some Abelian groups $\Gamma, \Gamma'$ where $|\Gamma| < |\Gamma'|$.
But maybe the existence of many $\Gamma$-colorings implies some (or even many) $\Gamma'$-colorings.
Jaeger et al. \cite{jlpt} proved that $\Z_5$-colorability does not imply $\Z_6$-colorability.
A planar graph has at least $(k-5)^k$ $\Z_k$-colorings. 
For $\Z_6$ we can do better:
By repeating the present proof we obtain at least $2^{n/9}$ $\Z_6$-colorings.
By using the proof in \cite{t1} we obtain even $(3/2)^n$ $\Z_6$-colorings.

In this paper, we will maintain the same structure as in \cite{t2}. Thus the theorems and lemmas, etc., will be given the same numbers, and those proofs from \cite{t2} which carry over will stand as in \cite{t2}.

\section{Definitions}

In this paper we consider simple planar graphs. We follow the notation of Mohar and Thomassen \cite{mt}.
Each edge in the graph will be given an orientation. The orientation will be fixed, but the specific orientation of a graph will not be important.

We will introduce an additional constraint on the group colorings of graphs by letting $F_v \subseteq \Z_5$ denote a set of \textit{forbidden} colors at the vertex $v$. We thus require a group coloring $c: V(G) \to \Z_5$ to satisfy $c(v) \notin F_v$.
Furthermore, we will use the notation $L_v = \Z_5 \setminus F_v$ to denote the set of \textit{available} colors at $v$.

In general we define group colorability as follows.

\begin{defi}
Let $\Gamma$ be an Abelian group. The graph $G$ is said to be {\em $\Gamma$-colorable} if the following holds: Given some orientation of $G$ and any function $\phi: E(G) \to \Gamma$ there exists a vertex coloring $c: V(G) \to \Gamma$ such that
$c(w) - c(u) \neq \phi(uw)$
for each $uw \in E(G)$ where $uw$ is directed towards $w$.
\end{defi}

If this holds we say that $c$ is \textit{proper} with respect to $\phi$.
If the function $\phi: E(G) \to \Gamma$ is given, we define $(\Gamma, \phi)$-colorability as follows:

\begin{defi}
$G$ is said to be {\em $(\Gamma, \phi)$-colorable} if there exists a vertex coloring $c: V(G) \to \Gamma$ such that
$c(w) - c(u) \neq \phi(uw)$
for each $uw \in E(G)$ where $uw$ is directed towards $w$.
\end{defi}

Note on notation: Formally, $\phi(uw)$ is defined on every directed edge $uw$. But, we also write $\phi(wu) = -\phi(uw)$.

The following property of group coloring will prove useful later.

\begin{prop} \label{prop:newphi}
Let $\phi: E(G) \to \Gamma$. Given $v_0 \in V(G)$ and $\alpha \in \Gamma$, we define $\phi': E(G) \to \Gamma$ as follows:
\begin{equation}
\phi'(e) =
\begin{cases}
\phi(e) + \alpha & \text{if $e$ is incident to $v_0$ and directed towards $v_0$}, \\
\phi(e) - \alpha & \text{if $e$ is incident to $v_0$ and directed away from $v_0$}, \\
\phi(e) & \text{otherwise.} \\
\end{cases}
\end{equation}
Then $G$ is $(\Gamma, \phi)$-colorable if and only if $G$ is $(\Gamma, \phi')$-colorable.
\end{prop}

\subsection{The function $\tau$}

In addition to the standard definitions in Section 2 above we introduce a collection of functions $\tau$ which, given a coloring of some vertex $v \in V(G)$, determines the colors at the neighboring vertices of $v$ that are not allowed by the coloring of $v$:

\begin{defi}
Given a function $\phi: E(G) \to \Z_5$ and a vertex $v$ with prescribed color $c(v)$ we will define the function $\tau_v: N(v) \to \Z_5$ to be: 
\begin{equation}
\tau_v(u) =
\begin{cases}
c(v) + \phi(uv) & \text{if $uv$ is directed towards $u$}, \\
c(v) - \phi(uv) & \text{if $uv$ is directed towards $v$}.
\end{cases}
\end{equation}
In case the coloring of $v$ is not prescribed we will define $\tau_v: \Z_5 \times N(v) \to \Z_5$ to be: 
\begin{equation}
\tau_v(\alpha, u) =
\begin{cases}
\alpha + \phi(uv) & \text{if $uv$ is directed towards $u$}, \\
\alpha - \phi(uv) & \text{if $uv$ is directed towards $v$}.
\end{cases}
\end{equation}
Furthermore, given $S \subseteq \Z_5$ we will define $\tau_v(S, u) := \{\tau_v(s, u) \mid s \in S\}$.
\end{defi}

Note, that $\tau$ is well-defined since $G$ is simple. Now, a $\Z_5$-coloring $c: V(G) \to \Z_5$ of $G$ is proper with respect to $\phi$ if and only if $c(u) \neq \tau_v(u)$ for all pairs of neighbors $v,u \in V(G)$.
Informally, if we give $v$ the color $\alpha$, then we cannot give $u$ the color $\tau_v(\alpha,u)$.

Note, that when using the $\tau$-function we will no longer need to specify the orientation of $G$.

\begin{obs} \label{obs1}
If $c(v) = \alpha$, then $\tau_u(\tau_v(u), v) = \alpha$ for $\alpha \in \Z_5$ (regardless of the value of $\phi(uv)$).
Similarly, $\tau_u(\tau_v(S,u), v) = S$ for $S \subseteq \Z_5$. 
This can also be expressed as $\tau_v(\alpha, u) = \beta$ if and only if $\tau_u(\beta, v) = \alpha$ for $\alpha,\beta \in \Z_5$, and $\tau_v(S_1, u) = S_2$ if and only if $\tau_u(S_2, v) = S_1$ for $S_1, S_2 \subseteq \Z_5$.
\end{obs}

The function $\tau$ can also be defined for the more general DP-colorings introduced by Dvo{\v r}{\'a}k and Postle in \cite{dp}, and Observation \ref{obs1} also holds in this more general setup whereas the following Proposition \ref{prop:tau} which is an important feature of group colorings does not.

\begin{prop} \label{prop:tau}
Given $u,v,w \in V(G)$ such that $uv, vw, uw \in E(G)$. If  
$\tau_v(\tau_u(\alpha,v),w) = \tau_u(\alpha,w)$ for some $\alpha \in \Z_5$, then it holds for any $\alpha \in \Z_5$. 
\end{prop}

\begin{proof}
We can assume without loss of generality that $uv$ is directed towards $v$, $vw$ is directed towards $w$, and $wu$ is directed towards $u$.
If there exists an $\alpha \in \Z_5$, such that $\tau_v(\tau_u(\alpha,v),w) = \tau_u(\alpha,w)$, then 
\begin{equation}
(\alpha + \phi(uv)) + \phi(vw) = \alpha - \phi(wu).
\end{equation} 
Hence $\phi(uv) + \phi(vw) + \phi(wu) = 0 \pmod{5}$. Thus $\tau_v(\tau_u(\alpha,v),w) = \tau_u(\alpha,w)$ for any $\alpha \in \Z_5$.
\end{proof}

\section{$\Z_5$-colorings with precolored vertices}

In the rest of this paper we assume $G$ is an oriented plane near-triangulation with outer cycle $C: v_1v_2 \ldots v_k v_1$. 

\begin{defi}
Given $\phi: E(G) \to \Z_5$ we say that $G$ is \textit{$(\Z_5,3)$-extendable} with respect to $\phi$ and the vertices $v_1, v_2, v_k$ if the following holds: 
Assume that the vertices $v_k$, $v_1$ and $v_2$ are precolored $c(v_k), c(v_1), c(v_2)$, respectively, such that $c(v_k) \neq \tau_{v_1}(v_k)$ and $c(v_2) \neq \tau_{v_1}(v_2)$, and for each $v \in C \setminus \{v_1, v_2, v_k\}$, $F_v$ is a set containing at most two forbidden colors. For all other vertices $v$, $F_v$ is empty.
Then $c$ can be extended to a $(\Z_5, \phi)$-coloring of $G$ which we also call $c$ and which satisfies $c(v) \notin F_v$ for any $v \in C \setminus \{v_1, v_2, v_k\}$.
\end{defi}

Note, that the analogous definition of \textit{$(\Z_5,2)$-extendability} is used to prove Theorem \ref{thm:A} above which can be phrased as follows:

\begin{thm} \label{thm:2-ext}
Any oriented near-triangulation is $(\Z_5,2)$-extendable with respect to any $\phi$-function and any path on two vertices on the outer cycle.
\end{thm}

This implies the following:

\begin{thm} \label{thm:shortcycle}
Let $\phi: E(G) \to \Z_5$ be given where $G$ is a near-triangulation with precolored outer cycle $C$ of length $k \leq 5$. 
Then $G$ has a $(\Z_5, \phi)$-coloring unless $C$ has length precisely 5, and $int(C)$ has a vertex $v$ joined to all vertices of $C$ such that $\{\tau_{v_1}(v), \ldots, \tau_{v_5}(v)\} = \Z_5$. 
\end{thm}

\begin{proof}
The proof is by induction on the number of vertices of $G$. 
If no vertex of $int(C)$ is joined to more than two vertices of $C$, then we consider the subgraph $H$ induced by the vertices in $int(C)$.
We let the set of forbidden colors of a vertex be the colors forbidden by its neighbors in $C$. By Theorem \ref{thm:2-ext}, $H$ is $(\Z_5, 2)$-extendable with these sets of forbidden colors. 
(If $H$ is not 2-connected, then we color the blocks of $H$ successively.)
So we may assume that some vertex $u$ has at least three neighbors in $C$. 
If it is not possible to color $u$, then $G$ satisfies the conclusion of Theorem \ref{thm:shortcycle}. 
On the other hand, if it is possible to color $u$, then we color it and complete the proof by induction by coloring the interior of each precolored cycle on the form $v_i  \cdots v_{i+j} u v_i$ (where $j=1,2,3$). The only case where this might not work is if there is some vertex $v$ in the interior of one of the colored cycles which is joined to all vertices in a colored 5-cycle, but then $u$ must have precisely three consecutive neighbors in $C$, and we therefore have two possibilities for coloring $u$. So, the exceptional case in Theorem \ref{thm:shortcycle} can be avoided.
\end{proof}

\section{Generalized wheels and generalized multi-wheels}

We define \textit{wheels}, \textit{broken wheels} and \textit{generalized wheels} as in \cite{t2}:
The outer cycle $C$ is of the form $v_1 v_2 \cdots v_k v_1$ where $v_1$ is the \textit{major vertex}, $v_k, v_2$ are \textit{principal neighbours}, $v_k v_1, v_1 v_2$ are \textit{principal edges}, and $v_k v_1 v_2$ is the \textit{principal path}.
If the interior of $C$ consists of the edges $v_1v_3, v_1v_4, \ldots, v_1v_{k-1}$, then we call $G$ a \textit{broken wheel}. 
If the interior of $C$ consists of a vertex $v$ and the edges $vv_1, vv_2, \ldots, vv_{k}$, then we call $G$ a \textit{wheel}. 
We define \textit{generalized wheels} to be the class of graphs containing all broken wheels and wheels, as well as the graphs obtained from two generalized wheels by identifying a principal edge in one of them with a principal edge in the other such that their major vertices are identified.

Note, that it is easy to see that a broken wheel with at least four vertices is not $(\Z_5,3)$-extendable with respect to $v_k, v_1, v_2$, and a wheel with an even number of (at least six) vertices is not $(\Z_5,3)$-extendable with respect to $v_k, v_1, v_2$.

In addition to these graphs we need a class of graphs which extends the wheels, as well as a class which extends the generalized wheels. 

We define an operation as follows:  Let $G$ be a generalized wheel and assume that $v_i, u, v_{i+1} \in V(G) \setminus \{v_1\}$ form a facial triangle where $v_i v_{i+1}$ is an edge on the outer cycle $C$ and $v_i u, v_{i+1} u$ are edges in $int(C)$. We obtain a new graph $G'$ from $G$ by adding a new vertex $w$ and the edges $uw, v_iw, v_{i+1}w$, as well as replacing the edge $v_i v_{i+1}$ by a path $v_i w_1 \cdots w_j v_{i+1}$ with $j \geq 0$ and adding the edges $w_1 w, \ldots, w_j w$. We say that we \textit{insert a wheel} into the triangle $v_i u v_{i+1}$.

\begin{defi}
We define \textit{multi-wheels} to be the class of graphs containing all wheels, as well as the graphs obtained from a multi-wheel by inserting a wheel into a triangle as above.
\end{defi}

\begin{defi} \label{defi:5}
We define \textit{generalized multi-wheels} to be the class of graphs containing all generalized wheels, as well as the graphs obtained from a generalized multi-wheel by inserting a wheel into a triangle as above.
\end{defi}

Note, that a broken wheel is also a generalized wheel (and therefore also a generalized multi-wheel), but a broken wheel is not a multi-wheel.

Observe, that if we replace the operation in Definition \ref{defi:5} by inserting generalized multi-wheels into triangles instead of inserting wheels, then we get the exact same class of graphs.

\begin{prop} \label{prop:genmulti}
Let $G$ be a generalized multi-wheel with outer cycle $C$. If $uvwu$ is a facial triangle with $u,v,w \in V(G)$, then at least one of $u,v,w$ is on $C - v_1$.
\end{prop}

\begin{proof}
The statement is clearly true for all facial triangles in wheels, broken wheels and generalized wheels. 
As the statement remains true whenever a wheel is inserted into a triangle, it is also true for multi-wheels and generalized multi-wheels.
\end{proof}

\begin{lemma} \label{lemma1}
Let $\phi: E(G) \to \Z_5$ be given where $G$ is a multi-wheel. Assume that for each $v \in \{v_3, v_4, \ldots, v_{k-1}\}$, $F_v$ is a forbidden set containing at most two colors  in $\Z_5$. For all other vertices $v$, $F_v$ is empty.
Then there exists $\alpha \in \Z_5$ such that the $(\Z_5, \phi)$-colorings of $v_k, v_1, v_2$ which cannot be extended to $G$ satisfy that $c(v_k) - c(v_2) = \alpha$.
\end{lemma}

It is easy to see such an $\alpha$ does not exist if $G$ is a broken wheel on 4 (and hence any larger number of) vertices. 
This may explain why the proof of Lemma \ref{lemma1} is not trivial.

\begin{proof}[Proof of Lemma 1.]
We prove Lemma \ref{lemma1} by induction on the number of vertices $n$. 
Assume $n \geq 5$ since otherwise there is nothing to prove.
Also, by Theorem \ref{thm:shortcycle} we can assume that $k \geq 5$.
Consider first the case where $G$ is a wheel. Let $v$ be the vertex not in $C$.
Suppose $v_k, v_1, v_2$  are colored $c(v_k), c(v_1), c(v_2)$, respectively, and that this coloring cannot be extended to $G$.
Construct $\phi': E(G) \to \Z_5$ from $\phi$ using Proposition \ref{prop:newphi} successively with $v_{k-1}, v_k, v_1, v_2, v_3$ respectively playing the role of $v_0$ such that $\phi'(v_{k-1}v) = \phi'(v_{k}v) = \phi'(v_{1}v) = \phi'(v_{2}v) = \phi'(v_{3}v) = 0$ with corresponding precoloring $c'(v_k), c'(v_1), c'(v_2)$ and tau-function $\tau'$.
It suffices to prove Lemma \ref{lemma1} with this $\phi'$ instead of $\phi$.
Now $\tau'_{v_i}(v) = c'(v_i)$ for $i \in \{k,1,2\}$, and similarly $\tau'_{v_i}(\alpha,v) = \alpha$ for $i \in \{3,k-1\}, \alpha \in \Z_5$.
Then $L_{v_3} \setminus \tau'_{v_2}(v_3)$ consists of precisely two colors of $\Z_5$, say $\alpha, \beta$, since otherwise we can color $v$ (with at least two color options) and extend that coloring to $G$ by applying Theorem \ref{thm:2-ext} to $G - v_1 - v_2$. 
Similarly, $L_{v_{k-1}} \setminus \tau'_{v_k}(v_{k-1})$ consists of precisely two colors, say $\gamma, \delta$. 
If $L_v \setminus \{c'(v_k), c'(v_1), c'(v_2)\}$ contains a color $\epsilon$ distinct from $\alpha, \beta$, then we can give $v$ that color, put $L_{v_3} = \{\alpha, \beta, \epsilon\}$, and then extend the resulting coloring to $G$ by applying Theorem \ref{thm:2-ext} to $G-v_1-v_2$, a contradiction.
So we may assume that $L_v \setminus \{c'(v_k), c'(v_1), c'(v_2)\} = \{\alpha, \beta\}$. In particular, $c'(v_k), c'(v_1), c'(v_2)$ are distinct.
Similarly, $L_v \setminus \{c'(v_k), c'(v_1), c'(v_2)\} = \{\gamma, \delta\}$.
Thus $L_{v_3}, L_{v_{k-1}}$ have at least two colors in common, namely $\alpha, \beta$.
Consider first the case where $L_{v_3}, L_{v_{k-1}}$ have precisely two colors in common. In this case we argue as in \cite{t2}:
$c'(v_2)$ is the unique color of $L_{v_3} \setminus L_{v_{k-1}}$, $c'(v_k)$ is the unique color of $L_{v_{k-1}} \setminus L_{v_3}$, and $c'(v_1)$ is the unique color of $L_v \setminus (L_{v_3} \cup L_{v_{k-1}})$.
This shows that the coloring of $v_1, v_2, v_k$ is unique.
Consider next the case where $L_{v_3}, L_{v_{k-1}}$ have more than two colors in common, that is, they are equal. In this case, assume without loss of generality that $v_2v_3$ is directed towards $v_3$ and $v_kv_{k-1}$ is directed towards $v_{k-1}$. 
We observe that $\tau'_{v_2}(v_3) = \tau'_{v_k}(v_{k-1})$, i.e. $c'(v_2) + \phi'(v_2v_3) = c'(v_k) + \phi'(v_kv_{k-1})$. Thus $c'(v_k) - c'(v_2) = \phi'(v_2v_3) -\phi'(v_kv_{k-1})$ can play the role of $\alpha$ in Lemma \ref{lemma1}.

Consider now the case where $G$ is a multi-wheel, but not a wheel. 
Recall, that $G$ is obtained from a multi-wheel by inserting a wheel into a triangle. More precisely, $G$ has a vertex $u$ in $int(C)$ joined to $v_i, v_{i+1}, \ldots, v_j$ and also joined to a vertex $v$ in $int(C)$ such that $v$ is joined to $v_i,v_j$. 
We may assume that $j>i+1$ since otherwise, we delete $u$ and complete the proof by induction.
Let $G'$ be the subgraph of $G$ which has outer cycle $C' = v v_i \cdots v_j v$. Note that $G'$ is a wheel.
By the induction hypothesis there exists $\alpha' \in \Z_5$ such that all colorings of $v_i, v, v_j$ which cannot be extended to $G'$ satisfy $c(v_j) - c(v_i) = \alpha'$.
Now use the induction hypothesis on the graph $G''$ obtained from $G$ by replacing $G'$ by the triangle $v v_i v_j v$ (with $v_iv_j$ directed towards $v_j$) where we define $\phi(v_i v_j) = \alpha'$. 
All colorings of $v_k, v_1, v_2$ that can be extended to $G''$ clearly also extends to $G$. Thus the colorings of $v_k,v_1,v_2$ that cannot be extended to $G$ satisfy the conclusion of Lemma \ref{lemma1} with the same $\alpha$ as the one we found for $G''$ using the induction hypothesis.
\end{proof}

Note, that $\alpha$ does not depend on $\phi(v_kv_1), \phi(v_1v_2)$. More precisely, if we let $\phi': E(G) \to \Z_5$ be a function that agrees with $\phi$ on all edges except $v_kv_1, v_1v_2$, then the $\alpha$ that works for $\phi$ also works for $\phi'$.

\begin{lemma} \label{lemma2}
Let $\phi: E(G) \to \Z_5$ be given where $G$ is a generalized multi-wheel with no separating triangles. Assume that each for each $v \in \{v_3, v_4, \ldots, v_{k-1}\}$, $F_v$ is a forbidden set containing at most two colors of $\Z_5$, and assume that $v_k, v_1, v_2$ are precolored. For all other vertices $v$, $F_v$ is empty.
Let $e$ be any edge in $E(G) \setminus \{v_kv_1, v_1v_2\}$. Then $G-e$ has a $(\Z_5, \phi)$-coloring $c: V(G) \to \Z_5$ that extends the precoloring and satisfies $c(v) \notin F_v$ for any $v \in C \setminus \{v_1, v_2, v_k\}$.
\end{lemma}

\begin{proof}
By induction on the number of vertices in $G$.
The statement is easy to verify if $G$ is a broken wheel.
Consider now the case where $G$ is a wheel. Consider the subcase $e = vv_i$ where $v$ is the vertex in $int(C)$. If $3 \leq i \leq k-1$ then we color $v, v_3, v_4, \ldots, v_{i-1}, v_{k-1}, v_{k-2}, \ldots, v_i$ in that order.
If $i \in \{k,1,2\}$ then we choose the color of $v$ such that $\tau_v(v_{k-1}) \notin L_{v_{k-1}} \setminus \{\tau_{v_k}(v_{k-1})\}$ (in case that set has precisely two colors) where $L_{v_{k-1}}$ denotes the list of available colors at $v_{k-1}$. Then we color $v_3, v_4, \ldots, v_{k-1}$ in that order.
If $e = v_i v_{i+1}$ is on $C$ then we color $v, v_3, v_4, \ldots, v_{i}, v_{k-1},$ $v_{k-2}, \ldots, v_{i+1}$ in that order.

Consider next the case where $G$ is a multi-wheel, but not a wheel.
Recall, that $G$ is obtained from a multi-wheel by inserting a wheel into a triangle. More precisely, $G$ has a vertex $u$ in $int(C)$ joined to $v_i, v_{i+1}, \ldots, v_j$ and also joined to a vertex $v$ in $int(C)$ such that $v$ is joined to $v_i,v_j$.
Now $j \geq i+2$ as $G$ has no separating triangles.
Let $G'$ be the subgraph of $G$ which has outer cycle $C' = v v_i \cdots v_j v$. Note that $G'$ is a wheel. 
Let $G''$ be the graph obtained from $G$ by replacing $G'$ by the triangle $vv_iv_jv$ with $v_iv_j$ directed towards $v_j$.
If $e$ is in $E(G'')$, then use Lemma \ref{lemma1} to obtain $\alpha \in \Z_5$ such that all colorings of $v_i, v, v_j$ which cannot be extended to $G'$ satisfy $c(v_j) - c(v_i) = \alpha$. The induction hypothesis implies that there exists a $(\Z_5, \phi)$-coloring of $G'' -e$ where we define $\phi(v_i v_j) = \alpha$.
By Lemma \ref{lemma1} this coloring can be extended to $G'$. 
(If $e$ is one of the two edges $vv_i, vv_j$, then we use the remark following the proof of Lemma \ref{lemma1}.) Thus we get a $(\Z_5, \phi)$-coloring of $G-e$.

If $e$ is not in $E(G'')$, then the induction hypothesis implies that $G''-e'$ is $(\Z_5, \phi)$-colorable where $e' = v_iv_j$. This coloring can be extended to $G'-e$, again using the induction hypothesis. Thus $G-e$ is $(\Z_5, \phi)$-colorable.

Assume now that $G$ contains a chord $v_1v_i$ and that $e$ is not $v_1v_i$.
Then $v_1 v_i$ divides $G$ into near-triangulations $G_1,G_2$ where $G_1$ has outer cycle $v_1 v_2 \cdots v_i v_1$ and $G_2$ has outer cycle $v_1 v_i v_{i+1} \cdots v_1$.
Assume without loss of generality that $e \in G_1$. Then $G_2$ can be colored by Theorem \ref{thm:2-ext}, and the induction hypothesis implies that the coloring can be extended to $G_1 - e$.

Assume finally that $G$ is the union of two multi-wheels $G_1, G_2$ and $e = v_1 v_i$ is their common edge. 
We may assume that $v_i$ has precisely three available colors since otherwise we delete one or two available colors.
By Theorem \ref{thm:2-ext} each of $G_1, G_2$ can be $(\Z_5, \phi)$-colored, and we get another coloring of each graph by using Theorem \ref{thm:2-ext} on $G_1, G_2$ where we define $\phi(v_1v_i)$ to be the color of $v_i$ minus the color of $v_1$ in the first coloring. Thus we have two colorings of each of $G_1, G_2$ where $v_i$ has distinct colors.
Combining the colorings of $G_1, G_2$ in which the color of $v_i$ is the same gives a $(\Z_5, \phi)$-coloring of $G-e$.
\end{proof}

\begin{lemma} \label{lemma3}
Let $\phi: E(G) \to \Z_5$ be given where $G$ is a near-triangulation. 
\begin{enumerate}[label=\alph*)]
\item
Assume that the interior of the outer cycle $C$ has precisely two vertices $u,v$, and that there exists a natural number $i$, $3 \leq i \leq k-1$, such that $u$ is joined to $v, v_1, v_2, \ldots, v_i$, and $v$ is joined to $u, v_i, v_{i+1}, \ldots, v_k, v_1$. Then $G$ is $(\Z_5, 3)$-extendable with respect to $\phi$ and the path $v_kv_1v_2$.

\item
Assume next that the interior of the outer cycle $C$ has precisely two vertices $u,v$, and that there exists a natural number $i$, $4 \leq i \leq k-1$, such that $u$ is joined to $v, v_2, \ldots, v_i$, and $v$ is joined to $u, v_1,v_2, v_i , v_{i+1}, \ldots, v_k$. Then $G$ is $(\Z_5, 3)$-extendable with respect to $\phi$ and the path $v_kv_1v_2$.
\end{enumerate}
\end{lemma}

\begin{proof}[Proof of a)]
Assume that $v_k, v_1, v_2$ are precolored.
Let $S_u = \Z_5 \setminus \{\tau_{v_1}(u), \tau_{v_2}(u)\}$, $S_v = \Z_5 \setminus \{\tau_{v_k}(v), \tau_{v_1}(v)\}$, and $S_i = L_{v_i}$.

We give $u$ a color from $S_u$, say $\alpha_u$, such that $L_{v_3} \setminus \{\tau_{v_2}(v_3), \tau_u(\alpha_u, v_3)\}$ contains at least two colors.
If $i=k-1$ then we color $v_{k-1}, v_{k-2}, \ldots, v_3, v$ in that order. So assume that $i \leq k-2$, and, similarly, $i \geq 4$.

If it is now possible to color $v$ such that $L_{v_i} \setminus \{\tau_v(v_i), \tau_{u}(v_i)\}$ has at least two colors, then it is easy to complete the coloring by coloring $v_{k-1}, v_{k-2}, \ldots, v_3$ in that order. 
So we may assume that such colorings of $u$ and $v$, respectively, are not possible.
Then we must have $|S_v| = |S_i| = 3$, so we let $S_v = \{\alpha_v, \beta_v, \gamma_v\}$, $S_i = \{\alpha_i, \beta_i, \gamma_i\}$.
In particular, after $u$ has received color $\alpha_u$, the colors $\tau_u(\alpha_u, v) =: \alpha_v$ and $\tau_u(\alpha_u, v_i) =: \alpha_i$ are no longer available at $v$ and $v_i$, respectively, and furthermore $\tau_v(\beta_v, v_i) =: \beta_i$ and $\tau_v(\gamma_v, v_i) =: \gamma_i$ are the remaining available colors at $v_i$. 

Similarly, we can choose a color $\delta$ from $S_v$, such that $L_{v_{k-1}} \setminus \{\tau_{v_k}(v_{k-1}), \tau_v(\delta, v_{k-1})\}$ contains at least two colors. If $\delta$ is not $\alpha_v$, then we may color $G$ by letting $u$ have color $\alpha_u$, $v$ have color $\delta$, and completing the coloring by coloring $v_i, v_{i-1}, \ldots, v_3, v_{i+1}, \ldots, v_{k-1}$ in that order. So we may assume that $\delta = \alpha_v$.
As above, we conclude $|S_u| = 3$, so we let $S_u = \{\alpha_u, \beta_u, \gamma_u\}$. 
In particular, after $v$ has received color $\alpha_v$, the colors $\tau_v(\alpha_v, u) = \alpha_u$ and $\tau_v(\alpha_v, v_i) = \alpha_i$ (the latter equality holds since we know $\tau_v(\beta_v, v_i) = \beta_i$ and $\tau_v(\gamma_v, v_i) = \gamma_i$) are no longer available at $u$ and $v_i$, respectively. Choose the notation for $\beta_u, \gamma_u$ such that $\tau_u(\beta_u, v_i) = \beta_i$ and $\tau_u(\gamma_u, v_i) = \gamma_i$ are the remaining available colors at $v_i$.
Using Proposition \ref{prop:tau}, we get that $\tau_u(\beta_u, v) = \beta_v$ and $\tau_u(\gamma_u, v) = \gamma_v$. Thus, on the triangle $uvv_iu$ any $(\Z_5, \phi)$-coloring must consist of one $\alpha$, one $\beta$ and one $\gamma$.

Now, we give $u$ the color $\alpha_u$. If  $\tau_v(\beta_v, v_{k-1})$ is not in $L_{v_{k-1}} \setminus \tau_{v_k}(v_{k-1})$ then we give $v$ color $\beta_v$ and we color $v_i, v_{i+1}, \ldots, v_{k-1}, v_{i-1}, \ldots, v_3$. 
So we may assume that $\tau_v(\beta_v, v_{k-1}),$ $\tau_v(\gamma_v, v_{k-1})$ are the only colors in $L_{v_{k-1}} \setminus \tau_{v_k}(v_{k-1})$.

We give $u$ the color $\alpha_u$, we give $v$ the color $\beta_v$, and we color $v_{k-1}, v_{k-2}, \ldots, v_{i+1}$. If this coloring can be extended to $v_i$, it is easy to complete the coloring by coloring $v_{i-1}, \ldots, v_3$.
So we may assume that $v_{i+1}$ has color $\tau_{v_i}(\gamma_i, v_{i+1})$.

We now try another coloring. We give $u$ the color $\gamma_u$, we give $v$ the color $\alpha_v$, and we color $v_3, v_4, \ldots, v_{i-1}$. We may assume that this coloring cannot be extended to $v_i$, that is, $v_{i-1}$ has color $\tau_{v_i}(\beta_i, v_{i-1})$.

Now we keep the colors of $v_3, v_4, \ldots, v_{i-1}, v_{i+1}, \ldots, v_{k-1}$ given above. And we give $u, v, v_i$ the colors $\gamma_u, \beta_v, \alpha_i$, respectively. 
This gives a $(\Z_5, \phi)$-coloring of $G$.
\end{proof}

\begin{proof}[Proof of b)]
Assume again that $v_k,v_1,v_2$ are precolored. 
We delete the precolored vertices and call the resulting graph $H$. 
Note that $v_3,v_{k-1},v$ each has at least two available colors, $u$ has at least four available colors, and each other vertex has at least three available colors. 
We complete the proof by induction on the number of vertices of $H$. 
 
Consider first the case where $i=k-1$, that is, $v$ has degree 2 in $H$. 
It is easy to see that we can give $u$ a color such that two of $v_3,v_{k-1},v$ still has at least two available colors. 
We then color the third of $v_3,v_{k-1},v$, and thereafter it is easy to color the remaining vertices one by one.

Consider next the case where $i=k-2$, that is, $v$ has degree 3 in $H$.
If possible, we give $v_{k-2}$ a color such that each of $v, v_{k-1}$ still has two available colors. 
Then we delete $v_{k-2}, v_{k-1}$ and can easily color the rest of $H$. 
So assume that such a coloring of $v_{k-2}$ is not possible.
If it is possible to color one of $v,v_{k-1}$ such that $v_{k-2}$ still has three available colors, then we color both of $v,v_{k-1}$ such that $v_{k-2}$ still has two available colors, we delete these two vertices, and then it is again easy to color the rest. 
So, we can assume that no such coloring of $v$ or $v_{k-1}$ is possible. 
Then $L_{v_{k-2}}$ contains a color $\alpha$ such that if we give $v_{k-2}$ the color $\alpha$, then each of $v,v_{k-1}$ has precisely one available color left. 
Now let $\beta, \gamma$ be two other colors in $L_{v_{k-2}}$.  
We choose the notation such that $\tau_{v_{k-2}}(\alpha, v) = \alpha_v$ and $\tau_{v_{k-2}}(\beta, v) = \beta_v$ where $\alpha_v, \beta_v \in L_v$. 
And we choose the notation such that $\tau_{v_{k-2}}(\alpha, v_{k-1}) = \alpha'$ and $\tau_{v_{k-2}}(\gamma, v_{k-1}) = \gamma'$ where $\alpha', \gamma' \in L_{v_{k-1}}$.
If we can give $v_{k-1}$ a color such that $\alpha_v,\beta_v$ are still available at $v$, then the proof reduces to the previous case where $v$ has degree 2. 
So, such a coloring of $v_{k-1}$ is not possible. 
Now we use Proposition \ref{prop:tau} to conclude that $\tau_{v_{k-1}}(\alpha', v) = \beta_v$ and $\tau_{v_{k-1}}(\gamma', v) = \alpha_v$. 
Now we delete $v_{k-1}$ and repeat the proof in the case where $v$ has degree 2. 
We let $v_{k-2}$ have the available colors $\alpha,\beta$. 
It is easy to see that the coloring of $H-v_{k-1}$ extends to $H$. 
 
Consider finally the case where $i<k-2$, that is, $v_{k-2}$ has degree 3 in $H$. 
We repeat the proof above with the exception that where we above after deleting vertices color the rest of the graph, we will in this case use induction on the remaining graph.
\end{proof}

\begin{cor} \label{cor:inner}
Assume $G$ is a multi-wheel with no separating triangle and with at least two inner vertices such that all inner vertices are joined to $v_2$. 
Then $G$ is $(\Z_5,3)$-extendable.
\end{cor}

\begin{proof}
$G$ has a unique path $v_1u_1u_2\cdots u_qv_2$, such that all of $u_1,u_2,\ldots,u_q$ are joined to $v_2$. 
The proof is by induction on $q$. If $q=2$, we use Lemma \ref{lemma3} b). 
So assume $q>2$. Since $G$ has no separating triangle, $v_3$ has degree 3. 
Now select two available colors in $L_{v_3}$, delete those colors from $L_{u_q}=\Z_5$, delete $v_3$ and all other neighbors of $u_q$ on $C$ of degree 3, and complete the proof by induction.
\end{proof}

\section{$(\Z_5,3)$-extendability}

As in \cite{t2} we now characterize the near-triangulations that are not $(\Z_5, 3)$-extendable.
Theorem \ref{thm:3} below is similar to Theorem 3 in \cite{t2} except that ``generalized wheel" in \cite{t2} is replaced by ``generalized multi-wheel". 
Figure \ref{fig:thm3} below shows an example of a graph which is a generalized multi-wheel but not a generalized wheel, whose precoloring does not extend to a coloring of the whole graph.

\begin{figure}[ht]
	\centering
	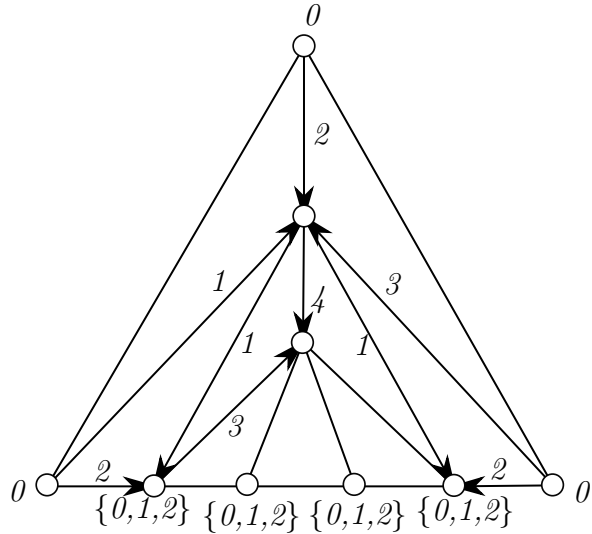
	\caption{A generalized multi-wheel whose precoloring $c(v_k) = c(v_1) = c(v_2) = 0$ does not extend to a coloring of the whole graph. (Note, that unlabelled edges $e$ have $\phi(e) = 0$.)}
	\label{fig:thm3}
\end{figure}

\begin{thm} \label{thm:3}
Let $\phi: E(G) \to \Z_5$ be given where $G$ is a plane near-triangulation with outer cycle $C: v_1 v_2 \cdots v_k v_1$. 
Assume that the vertices $v_k$, $v_1$ and $v_2$ are precolored, and for each $v \in C \setminus \{v_1, v_2, v_k\}$, $F_v$ is a set containing at most two forbidden colors. For all other vertices $v$, $F_v$ is empty.
Then $G$ has a $(\Z_5, \phi)$-coloring $c: V(G) \to \Z_5$ which extends the precoloring of $v_k, v_1, v_2$ and which satisfies $c(v) \notin F_v$ for any $v \in C \setminus \{v_1, v_2, v_k\}$, unless $G$ contains a subgraph $G'$ which is a generalized multi-wheel whose principal path is $v_kv_1v_2$, and all other vertices on the outer cycle of $G'$ are on $C$ and have precisely two forbidden colors.
\end{thm}

\begin{proof}
The proof is by induction on the number of vertices of $G$. 
For $k \leq 5$ the theorem follows from Theorem \ref{thm:shortcycle}.
So assume that $k>5$. Suppose for contradiction that the theorem is false, and let $G$ be a smallest counterexample.

\begin{claim} \label{claim1}
$C$ has no chord.
\end{claim} 

\begin{proof}
Suppose for contradiction that $v_i v_j$ is a chord of $C$, where $1 \leq j < i \leq k$. Then $v_i v_j$ divides $G$ into near-triangulations $G_1, G_2$, respectively. 
If $G_2$, say, does not contain $v_1$ then any $(\Z_5, \phi)$-coloring of $v_i v_j$ can be extended to $G_2$ by Theorem \ref{thm:2-ext}.
Therefore $G_1$ has no $(\Z_5, \phi)$-coloring. Now we apply the induction hypothesis to $G_1$ and obtain a contradiction. 
So assume that $j = 1$.

By Theorem \ref{thm:2-ext}, $G_2$ has a $(\Z_5, \phi)$-coloring. That coloring cannot be extended to $G_1$. The induction hypothesis implies implies that $G_1$ satisfies the conclusion of Theorem \ref{thm:3}, that is, $G_1$ contains a generalized multi-wheel. A similar argument shows that $G_2$ satisfies the conclusion of Theorem \ref{thm:3}. Thus $G$ contains a generalized multi-wheel.
It only remains to be proved that $L_{v_i}$ has only three available colors.
But if $L_{v_i} \setminus \{\tau_{v_1}(v_i)\}$ has a subset consisting of three colors, then, by Theorem \ref{thm:2-ext}, each of $G_1, G_2$ can be $(\Z_5, \phi)$-colored, and the color of $v_i$ can be chosen in two distinct ways (among these three colors) for each of $G_1, G_2$, since we get one coloring $c_1$ from Theorem \ref{thm:2-ext} and we get another coloring by replacing $c_1(v_i)$ by $\tau_{v_1}(v_i)$ in $L_{v_i}$.
Hence $G$ can be $(\Z_5, \phi)$-colored, a contradiction which proves Claim \ref{claim1}.
\end{proof}

\begin{claim} \label{claim2}
$G$ has no separating triangle and no separating 4-cycle.
\end{claim}

\begin{proof}
Suppose for contradiction that $G$ has a separating cycle $C'$ of length 3 or 4. 
We consider first the case where $C'$ has length 3. Delete $int(C')$ and  denote the resulting graph by $G'$. If $G'$ can be $(\Z_5, \phi)$-colored, then so can $G$ by Theorem \ref{thm:shortcycle}. 
So we may assume that $G'$ cannot be $(\Z_5, \phi)$-colored. Then $G'$ contains a generalized multi-wheel by the induction hypothesis, hence $G$ contains such a generalized multi-wheel, a contradiction.

We consider next the case where $C'$ has length 4. Choose $C'$ such that $int(C')$ is maximal. Replace $int(C')$ by a single edge $e$ and  denote the resulting graph by $G'$. If $G'$ can be $(\Z_5, \phi)$-colored, then so can $G$ by Theorem \ref{thm:shortcycle}. 
So we may assume that $G'$ cannot be $(\Z_5, \phi)$-colored. Then $G'$ contains a generalized multi-wheel satisfying the conclusion of Theorem \ref{thm:3} by the induction hypothesis. This generalized multi-wheel contains $e$ because we previously assumed that $G$ does not contain such a generalized multi-wheel.
The maximality property of $C'$ implies that $e$ is not contained in a separating triangle of $G'$. Then the first part of Claim \ref{claim2} implies that $G'$ has no separating triangles at all. 
So, if we delete the edge $e$ from $G'$, then the resulting graph can be $(\Z_5, \phi)$-colored by Lemma \ref{lemma2}. By Theorem \ref{thm:shortcycle}, $G$ can be $(\Z_5, \phi)$-colored, a contradiction which proves Claim \ref{claim2}.
\end{proof}

\begin{claim} \label{claim3}
If $u$ is a vertex in $int(C)$ which is joined to both $v_i,v_j$, where $2 \leq i \leq j-2 \leq k-2$, then $u$ is joined to each of $v_i, v_{i+1}, \ldots, v_j$.
\end{claim}

\begin{proof}
Suppose for contradiction that there exist $i',j'$ such that $i \leq i' \leq j' - 2 \leq j - 2$ and $u$ is joined to $v_{i'}, v_{j'}$, but not joined to any of $v_{i'+1}, v_{i'+2}, \ldots, v_{j'-1}$.
Let $C'$ be the cycle $u v_{i'} v_{i'+1} \cdots v_{j'} u$, and let $C''$ be the cycle $u v_{j'} v_{j'+1} \cdots v_k v_1 v_2 \cdots v_{i'} u$. 
We apply the induction hypothesis, first to $C'' \cup int(C'')$ and then to $C' \cup int(C')$. 
If $C' \cup int(C')$ is a generalized multi-wheel, then it is necessarily a multi-wheel, and then, by Lemma \ref{lemma1}, there exists $\alpha \in \Z_5$ such that all colorings of $v_{i'}, u, v_{j'}$ which cannot be extended to $G'$ satisfy $c(v_{j'}) - c(v_{i'}) = \alpha$.
So before we apply the induction hypothesis to $C'' \cup int(C'')$ we add the edge $v_{i'}v_{j'}$ and we let $\phi(v_{i'}v_{j'}) = \alpha$. 
Applying the induction hypothesis to this graph and then to $C' \cup int(C')$ either results in a $(\Z_5, \phi)$-coloring of $G$, hence we get a contradiction which proves Claim \ref{claim3}, or else we conclude that $C'' \cup int(C'') \cup \{v_{i'}v_{j'}\}$ contains a generalized multi-wheel satisfying the conclusion of Theorem \ref{thm:3}. 
This must contain the triangle $uv_{i'}v_{j'}u$ because of Claim \ref{claim1} and the assumption that no vertex on the outer cycle of the generalized multi-wheel has more than three available colors, and as $C' \cup int(C')$ is a multi-wheel we conclude that $G$ contains a generalized multi-wheel, a contradiction.
\end{proof}

\begin{claim} \label{claim4}
$G$ has no vertex in $int(C)$ which is joined to both $v_2$ and $v_k$.
\end{claim}

\begin{proof}
Suppose for contradiction that some vertex $u$ in $int(C)$ is joined to both $v_2$ and $v_k$. By Claim \ref{claim3}, $u$ is joined to all vertices of $C$ except possibly $v_1$. However, Claim \ref{claim2} implies that $u$ is joined to $v_1$, too.
Hence $G$ contains a spanning wheel. By Claim \ref{claim2}, $G$ is a wheel. If some vertex of $C$ has more than three available colors, then it is easy to $(\Z_5, \phi)$-color $G$. This contradiction proves Claim \ref{claim4}.
\end{proof}

\begin{claim} \label{claim5}
$v_3$ has degree at least 4.
\end{claim}

\begin{proof}
Suppose for contradiction that $v_3$ has degree at most 3. By Claim \ref{claim1}, $v_3$ has degree precisely 3, and $G$ has a vertex $u$ in $int(C)$ joined to $v_2, v_3, v_4$.
Let $i$ be the largest number such that $u$ is joined to $v_i$. 
The path $v_2 u v_i$ divides $G$ into two near-triangulations $G_1, G_2$ where $G_1$ contains $v_1$.
By Claims \ref{claim2} and \ref{claim3}, $G_2$ is a broken wheel. By Claim \ref{claim4}, $i<k$.

Now we use the argument of the proof of Theorem 1 in \cite{t1}.
We add to $F_u$ two colors of $\tau_{v_3}(L_{v_3} \setminus \tau_{v_2}(v_3), u)$.
We may assume that $G_1$ has no $(\Z_5, \phi)$-coloring under this assignment of forbidden sets. For otherwise, that coloring could be extended to $G - v_3$ and hence also to $G$.
Therefore the induction hypothesis implies that $G_1$ contains a generalized multi-wheel satisfying the conclusion of Theorem \ref{thm:3}. By Claims \ref{claim1} and \ref{claim2}, $G_1$ is a generalized multi-wheel. 

Claim \ref{claim4} implies that $G_1$ is not a multi-wheel. So $G_1$ has a chord. Claims \ref{claim1} and \ref{claim2} imply that this chord must be $v_1u$ and, since the chord is unique, $G_1-v_2$ is a multi-wheel. 
Claim \ref{claim3} then implies that all inner vertices of $G_1-v_2$ are joined to $u$. 
If $G_1-v_2$ is a wheel, then $G$ satisfies the assumption of Lemma \ref{lemma3} a) which implies that $G$ has a $(\Z_5, \phi)$-coloring, a contradiction. 
On the other hand, if $G_1-v_2$ is not a wheel, then we color $u$ such that $v_3$ still has two available colors. 
Corollary \ref{cor:inner} (applied to $G$ minus all those neighbors of $u$ that have degree 3) now implies that $G$ has a $(\Z_5, \phi)$-coloring, a contradiction which proves Claim \ref{claim5}.
\end{proof}

By a similar argument we get

\begin{claim} \label{claim6}
$v_{k-1}$ has degree at least 4.
\end{claim}

We now claim that

\begin{claim} \label{claim7}
$v_3$ and $v_{k-1}$ both have degree precisely 4, and $v_3$ and $v_{k-1}$ have a common neighbor in $int(C)$.
\end{claim}

\begin{proof}
Suppose for contradiction that Claim \ref{claim7} is false. Let $v_2, u_1, \ldots, u_q, v_4$ be the neighbors of $v_3$ in clockwise order. Then $q \geq 2$, by Claim \ref{claim5}.
Let $v_k, u'_1, \ldots, u'_q, v_{k-2}$ be the neighbors of $v_{k-1}$ in anti-clockwise order. Then $q' \geq 2$, by Claim \ref{claim6}. 
Let $u_iv_j$ be the unique edge such that $i$ is minimum and $j$ is maximum. 
By Claims \ref{claim2} and \ref{claim3}, $i=q$, and $j \leq k-2$. 
As in the proof of Claim \ref{claim5}, for each $1 \leq i \leq q$ we add to $F_{u_i}$ two colors of $\tau_{v_3}(L_{v_3} \setminus \tau_{v_2}(v_3), u_i)$.
And as in the proof of Claim \ref{claim5}, we conclude that $G-v_3$ contains a generalized multi-wheel $G'$.
By Claims \ref{claim1} and \ref{claim2}, we conclude that the outer cycle of this generalized multi-wheel $G'$ must be $C': v_1v_2u_1...u_qv_j...v_kv_1$. 
By Claim \ref{claim4}, $G'$ cannot be a multi-wheel (because every multi-wheel has a vertex in the interior joined to all three vertices of the principal path). 
So $G'$ has a chord. 
By Claims \ref{claim1} and \ref{claim2}, there can be only one chord, namely $v_1u_1$. As $C'': v_1u_1...u_qv_j...v_kv_1$ has no chord, it follows that $C''$  together with its interior is a multi-wheel which we call $G_1$. 
Then $int(C'')$ contains a vertex $v$ joined to all vertices of the principal path $v_kv_1u_1$ of $G_1$. 
By an analogous argument (with $v_{k-1}$ instead of $v_3$) we conclude that there exists a vertex $w$ joined to $v_2,v_1$ and a neighbor of $v_{k-1}$.
The only possibilities for $v,w$ are: $w=u_1$ and $v=u'_1$.
We now give $u_1$ a color such that $L_{v_3} \setminus \tau_{v_2}(v_3)$ still has at least two available colors. 
We delete $v_2$ and call the resulting graph $G_2$. 
By repeating the arguments given for $G_1$ above, we see that $G_2$ is also a multi-wheel.
If we apply Proposition \ref{prop:genmulti} to the triangle containing the edge $u'_1u'_2$ (but not the vertex $v_{k-1}$) in $G_2$, then we conclude that $u_1$ is joined to $u'_2$. 
Similarly, $u'_1$ is joined to $u_2$. 
This is possible only if $q=q'=2$ and $u_2=u'_2$. This contradiction proves Claim \ref{claim7}. 
\end{proof}

We are now ready for the final contradiction.
Using the proof of Claim \ref{claim7} we obtain the structure of $G'$: $u_1, u_1'$ and $u_2$ ($=u_2'$) are the only vertices in $int(C)$, $v_1$ is joined by an edge to $u_1$ and $u'_1$, and $u_1$ and $u'_1$ are joined by an edge.
By Claim \ref{claim3}, the cycle $u_2v_3v_4 \cdots v_{k-1}u_2$ together with its interior is a broken wheel. 
Define $G_2$ as in the proof of Claim \ref{claim7}. 
By Corollary 1, $G_2$ is colorable, and by the construction of $G_2$, this coloring can be extended to $G$.
We conclude that $G$ is $(\Z_5, \phi)$-colorable, a contradiction which completes the proof of Theorem \ref{thm:3}.
\end{proof}

\section{Further $\Z_5$-coloring properties of generalized multi-wheels}

\begin{lemma} \label{lemma4}
Let $G$ be a generalized multi-wheel, and let $\phi: E(G) \to \Z_5$.
Assume that the vertex $v_2$ is precolored, and for each $v \in C \setminus \{v_1, v_2\}$, $F_v$ is a set containing at most two forbidden colors. For all other vertices $v$, $F_v$ is empty.
Then it is possible to color $v_k$ such that any coloring of $v_1$ (introducing no color conflict with $v_2, v_k$) can be extended to a $(\Z_5, \phi)$-coloring $c: V(G) \to \Z_5$ of $G$ which satisfies $c(v) \notin F_v$ for any $v \in C \setminus \{v_1, v_2, v_k\}$.
\end{lemma}

\begin{proof}
We prove Lemma \ref{lemma4} by induction on the number of vertices of $G$.

If $G$ is a multi-wheel, then Lemma \ref{lemma4} follows easily from Lemma \ref{lemma1}.
Assume that $G$ is a generalized multi-wheel, but not a multi-wheel or a broken wheel. 
Then there exist $2 \leq i \leq j \leq k$ such that $G$ contains the edges $v_1v_i, v_1v_j$, and the cycle $C' := v_1 v_i v_{i+1} \cdots v_j v_1$ and its interior form a multi-wheel. 
Let $G' := C' \cup int(C')$, and let $G''$ be the graph obtained from $G$ by replacing $G'$ by a triangle $v_1 v_i v_j v_1$ (with $v_iv_j$ directed towards $v_j$). 
By Lemma \ref{lemma1} there exists $\alpha \in \Z_5$ such that all colorings of $v_{i}, v_1, v_{j}$ which cannot be extended to $G'$ satisfy $c(v_{j}) - c(v_{i}) = \alpha$.
Apply the induction hypothesis to $G''$ where we define $\phi(v_iv_j) = \alpha$. Then the resulting coloring can be extended to $G'$ by Lemma \ref{lemma1}, hence Lemma \ref{lemma4} follows.

So we can assume that $G$ is a broken wheel.
In particular, $v_1$ is joined to $v_3$. 
By Proposition \ref{prop:newphi} we may assume that $\phi(v_1v_i) = 0$ for each $2 \leq i \leq k$.
Let $\alpha, \beta$ be two colors in $L_{v_3} \setminus \{\tau_{v_2}(v_3)\}$.
Let $\gamma, \delta, \epsilon$ be three colors in $L_{v_k}$. 
Suppose for contradiction that for each of these three colors it is possible to color $v_1$ such that the coloring cannot be extended to $G$.
The color at $v_1$ must be one of $\alpha,\beta$, since otherwise the coloring can be extended by Theorem \ref{thm:2-ext} applied to $G-v_2$.
So for two of the colors, $\gamma, \delta, \epsilon$, say $\gamma, \delta$, it is the same color, say $\alpha$, which is used at $v_1$.
But now we get a contradiction to Theorem \ref{thm:2-ext} applied to $G$, where $v_1$ has the color $\alpha$, and $v_k$ has the available colors $\alpha, \gamma, \delta$.

This completes the proof of Lemma \ref{lemma4}.
\end{proof}

We define \textit{generalized wheel strings}, \textit{clean vertices}, and \textit{broken wheel strings} as in \cite{t2}:
If $G_1, G_2, \ldots, G_m$ are generalized wheels, then we define a \textit{generalized wheel string} by identifying each principal neighbor of the major vertex in $G_i$ with precisely one principal neighbor of the major vertex in $G_{i-1}, G_{i+1}$, respectively, for $i = 2, \ldots, k-1$. 
The two vertices which are principal neighbors of the major vertices in $G_1$ and $G_m$, respectively, and which have not been identified with any other vertex, are called \textit{clean vertices}.
If each of $G_1, G_2, \ldots, G_m$ is a broken wheel, then $G$ is a \textit{broken wheel string}.

We extend the first of these definitions as follows:

\begin{defi}
Let $G_1, G_2, \ldots, G_m$ be generalized multi-wheels. 
We define a \textit{generalized multi-wheel string} by identifying each principal neighbor of the major vertex in $G_i$ with precisely one principal neighbor of the major vertex in $G_{i-1}, G_{i+1}$, respectively, for $i = 2, \ldots, k-1$. 
\end{defi}

Given a generalized multi-wheel string, we now extend the definition of a \textit{clean vertex} to be the principal neighbors of the major vertices in $G_1$ and $G_m$ which have not been identified with any other vertex. 

\begin{lemma} \label{lemma5}
Let $G$ be a generalized multi-wheel string, and let $\phi: E(G) \to \Z_5$. 
Assume that the two clean vertices have forbidden sets containing at most three colors each, and that each non-clean vertex on the outer boundary has a forbidden set containing at most two colors. For all other vertices $v$, $F_v$ is empty.
Then it is possible to color the two clean vertices and all the cutvertices of $G$ such that any coloring of the major vertices (introducing no color conflict) can be extended to a $(\Z_5, \phi)$-coloring $c: V(G) \to \Z_5$ of $G$ which satisfies $c(v) \notin F_v$ for all $v \in V(G)$.
\end{lemma}

\begin{proof}
We prove Lemma \ref{lemma5} by induction on the number of vertices of $G$. Suppose for contradiction that $G$ is a smallest counterexample.

Let $G$ consist of the generalized multi-wheels $G_1, \ldots, G_m$ such that a principal neighbor of each of the major vertices in $G_i$ and $G_{i+1}$ are identified. 
Consider first the case where $m \geq 2$. 
Let $x$ (repectively $y$) be the clean vertex in $G_1$ (respectively $G_m$).
Let $z$ be the common vertex of $G_1$ and $G_2$.
Assume that $L_z = \{\alpha, \beta, \gamma\}$.
We now apply the induction hypothesis to $G_1$. We may assume that $x,z$ can be colored such that the conclusion of Lemma \ref{lemma5} holds.
Assume that the color of $z$ is $\alpha$. Then we again apply the induction hypothesis to $G_1$ but now we only allow colors $\beta, \gamma$ at $z$.
So the coloring of $x,z$ can be chosen in two ways in which $z$ has two distinct colors. 
Applying the induction hypothesis to $G_2 \cup \ldots \cup G_m$ there are two distinct colorings of $z,y$ (with $z$ getting different colors) such that the conclusion of Lemma \ref{lemma5} holds.
Now we let $z$ receive a color that appears in both a coloring of $x,z$ and a coloring of $y,z$.
So we may assume that $m=1$.

Let $x = v_2$, $y = v_k$ be the clean vertices in $G_1 = G$.
Assume first that $G$ is a generalized multi-wheel, but not a multi-wheel or a broken wheel. 
Then there exist $2 \leq i \leq j \leq k$ such that $G$ contains the edges $v_1v_i, v_1v_j$, and the cycle $C' := v_1 v_i v_{i+1} \cdots v_j v_1$ and its interior form a multi-wheel. 
Let $G' := C' \cup int(C')$, and let $G''$ be the graph obtained from $G$ by replacing $G'$ by a triangle $v_1 v_i v_j v_1$ (with $v_iv_j$ directed towards $v_j$). 
By Lemma \ref{lemma1} there exists $\alpha \in \Z_5$ such that the $(\Z_5, \phi)$-colorings of $v_j, v_1, v_i$ which cannot be extended to $G'$ satisfy that $c(v_j) - c(v_i) = \alpha$.
Apply the induction hypothesis to $G''$ where we define $\phi(v_iv_j) = \alpha$. Then the resulting coloring can be extended to $G'$ by Lemma \ref{lemma1}, hence Lemma \ref{lemma4} follows in the case where $G$ is a generalized multi-wheel, but not a multi-wheel and not a broken wheel.

So we can assume that $G$ is either a multi-wheel or a broken wheel.
We may assume that $G$ is a broken wheel since otherwise Lemma \ref{lemma5} follows easily from Lemma \ref{lemma1}.
By Proposition \ref{prop:newphi} we may assume that $\phi(v_1v_i) = 0$ for any $3 \leq i \leq k-1$, and also $\phi(v_2v_3) = \phi(v_{k-1}v_k) = 0$.
Furthermore, we may assume that all vertices $v_3, \ldots, v_{k-1}$ have precisely three available colors, since otherwise it is easy to see that any coloring of $v_2,v_k,v_1$ can be extended.
Let $L_{v_3} = \{\alpha, \beta, \gamma\}$, let $L_{v_2} = \{\alpha', \beta'\}$, and let $L_{v_k} = \{\alpha'', \beta''\}$.
Now, there are four possible ways of coloring $v_2, v_k$.
We may assume that none of them works, that is, for each of those four colorings, it is possible to color $v_1$ (introducing no color conflict with $v_2, v_k$) such that the resulting coloring cannot be extended to $G$. We say that these colors are the \textit{bad colors} of $v_1$.
Any bad color of $v_1$ must be in $L_{v_i}$ for each $3 \leq i \leq k-1$, since otherwise we color $v_3, \ldots, v_{i-1}, v_{k-1}, \ldots, v_i$ in that order.
In particular, the bad colors are among $\{\alpha, \beta, \gamma\}$.
Thus, for at least two of the four possibilities, $v_1$ has the same bad color, say $\gamma$. 
The two possibilities must either be $\alpha', \beta''$ and $\beta', \alpha''$ or $\alpha', \alpha''$ and $\beta', \beta''$, since otherwise (if $\beta'$, say, does not appear here) we apply Theorem \ref{thm:2-ext} to $G$ with $v_1, v_2$ colored $\gamma, \alpha'$, respectively, and $L_{v_k} = \{\alpha'', \beta'', \tau_{v_1}(v_k)\}$ to get a contradiction. 
Assume without loss of generality that the colorings $\alpha', \beta'', \gamma$ and $\beta', \alpha'', \gamma$ of $v_2, v_k, v_1$, respectively, cannot be extended to $G$.
The same argument shows that $\gamma$ cannot be the bad color of $v_1$ in three of the four possibilities.
We shall now argue that $\{\alpha', \beta'\} = \{\alpha, \beta\}$: If we give $v_1$ color $\gamma$ and $v_k$ color $\alpha''$ and then color $v_{k-1}, v_{k-2}, \ldots, v_4$ in that order, then the color at $v_3$ will be either $\alpha$ or $\beta$, say $\alpha$. 
If $\alpha$ is not in $\{\alpha', \beta'\}$, then $G$ is colorable with $v_1$ having color $\gamma$, a contradiction. 
If we next give $v_1$ color $\gamma$ and $v_k$ color $\beta''$, then the same argument implies that $\beta$ is in $\{\alpha', \beta'\}$. 
(If we have any choices while coloring $v_{k-1}, v_{k-2}, \ldots, v_4$, then it is easy to see that $G$ is colorable with $v_1$ having color $\gamma$, a contradiction. Thus $\beta$ must be the available color at $v_3$ when $v_4$ has been colored.) 
We choose the notation such that $\alpha' = \alpha$ and $\beta' = \beta$.
Now, if, say, $\beta$ is not a bad color of $v_1$ (that is, $\alpha$ and $\gamma$ are the only bad colors), then the coloring $\alpha, \alpha'', \alpha$ of $v_2, v_k, v_1$ does not extend to $G$, which gives a contradiction when we color $v_{k-1}, v_{k-2}, \ldots, v_3$ in that order ($v_3$ can be colored since $v_1, v_2$ have the same color and $\phi(v_1v_3) = \phi(v_2v_3) = 0$). 
Thus the colorings $\alpha, \alpha'', \beta$ and (by a similar argument) $\beta, \beta'', \alpha$ of $v_2, v_k, v_1$ do not extend to $G$, and since $\alpha, \beta, \gamma$ are all bad colors of $v_1$ we conclude $L_{v_3} = L_{v_4} = \ldots = L_{v_{k-1}} = \{\alpha, \beta, \gamma\}$ by an observation made earlier.
As above, we conclude that $\{\alpha'', \beta''\} = \{\alpha, \beta\}$. So $v_2$ and $v_k$ have the same available colors, namely $\alpha, \beta$.

Recall that none of the colorings $\alpha, \beta'', \gamma$ and $\alpha, \alpha'', \beta$ of $v_2, v_k, v_1$ extend to $G$.
We shall now obtain a contradiction by proving that at least one of them extends to $G$.
To prove this, let us now color $v_2, v_k, v_1$ by $\alpha, \beta'', \gamma$. 
Then we color $v_3, \ldots, v_{k-1}$ in that order according to the following rule: if $v_{i-1}$ has color $c(v_{i-1})$ and $\phi(v_{i-1}v_i) \neq 0$ then we give $v_{i}$ color $c(v_{i-1})$, and if $\phi(v_{i-1}v_i) = 0$ then we give $v_{i}$ the other available color (which is $\alpha$ or $\beta$) for $3\leq i \leq k-1$. 
Since the coloring of $v_2, v_k, v_1$ does not extend, we must get a color conflict between $v_{k-1}$ and $v_k$, that is, the color of $v_{k-1}$ is that of $v_k$ (since $\phi(v_{k-1}v_k) = 0$), that is, $\beta''$. 
If this happens, we let the colors of $v_2, v_k, v_1$ be $\alpha, \alpha'', \beta$ and now we give all vertices in $v_3, \ldots, v_{k-1}$ which have color $\beta$ color $\gamma$ instead.
This coloring clearly works.
\end{proof}

\section{Exponentially many $\Z_5$-colorings of planar graphs}

In this section we prove the main result. The proof follows closely the analogous proof in \cite{t2}.

\begin{thm} \label{thm:main}
Let $G$ be an oriented plane near-triangulation with outer cycle $C: v_1v_2\cdots v_kv_1$, and let $\phi: E(G) \to \Z_5$. 
For each vertex $v$ in $G$ let $F_v$ be a set of forbidden colors.
Assume that the vertices $v_k, v_1, v_2$ or the vertices $v_1, v_2$ are precolored.
If $v$ is one of $v_3, v_4, \ldots, v_{k-1}$ (resp. $v_3, v_4, \ldots, v_{k}$), then $F_v$ consists of at most two colors. For all other vertices $v$, $F_v$ is empty.
Let $n$ denote the number of non-precolored vertices, and let $r$ denote the number of vertices with precisely three available colors. 
Assume that $G$ has a $(\Z_5, \phi)$-coloring $c: V(G) \to \Z_5$ which satisfies $c(v) \notin F_v$ for any $v \in V(G)$. Then the number of such $(\Z_5, \phi)$-colorings is at least $2^{n/9-r/3}$, unless $G$ has three precolored vertices and also contains a vertex $u$ with precisely four available colors which is joined to the three precolored vertices and has only one available color distinct from $\tau_{v_k}(u), \tau_{v_1}(u), \tau_{v_2}(u)$.
\end{thm}

\begin{proof}
The proof is by induction on $n$. It is easy to verify the statement if $n = 1$ so we proceed to the induction step.
Let $f$ denote the number of vertices with precisely four available colors. 

We assume that $G$ is a counterexample such that $n$ is minimum and, subject to this, $r$ is maximal, and, subject to these conditions, $f$ is minimum.
We shall establish a number of properties of $G$ which will lead to a contradiction.
Clearly, $n > 3r$.

\begin{claim} \label{claim8}
$G$ has no separating triangle.
\end{claim}

\begin{proof}
Suppose for contradiction that $xyzx$ is a separating triangle which divides $G$ into near-triangulations $G_1, G_2$, respectively, where $G_1$ contains $C$.
Then any $(\Z_5, \phi)$-coloring of $x,y,z$ can be extended to $G_2$ by Theorem \ref{thm:shortcycle}.
Let $n_1$ be the number of non-precolored vertices in $G_1$, and let $n_2$ be the number of vertices in $G_2-x-y-z$.
By the minimality of $n$, $G_1$ has at least $2^{n_1/9-r/3}$ distinct $(\Z_5,\phi)$-colorings. Each such coloring has at least $2^{n_2/9}$ extensions to $G_2$.
As $n_1 + n_2 = n$, this proves Claim \ref{claim8}. 
\end{proof}

\begin{claim} \label{claim9}
$G$ has no chord.
\end{claim}

\begin{proof}
Suppose for contradiction that $v_iv_j$ is a chord of $C$, where $1 \leq i < j \leq k$. Then $v_iv_j$ divides $G$ into near-triangulations $G_1, G_2$, respectively.

Consider first the case where $G_2$, say, does not contain a precolored vertex distinct from $v_i,v_j$.
Then any $(\Z_5, \phi)$-coloring of $G_1$ can be extended to $G_2$ by Theorem \ref{thm:2-ext}.
We now obtain a contradiction by repeating the proof of Claim \ref{claim8}.

Assume next that $i = 1$ and that $v_k$ is precolored.
If each of $G_1, G_2$ is a generalized multi-wheel such that each non-precolored vertex on the outer cycle has precisely three available colors, then $r \geq n/3$, and there is nothing to prove.
So assume that $G_2$, say, is not such a generalized multi-wheel.
Moreover, it does not contain such a generalized multi-wheel because $G$ has no separating triangles, by Claim \ref{claim8}, and every chord of $G$, if any, is incident with $v_1$, by the first part of the proof of Claim \ref{claim9}. 
Now, if $j < k-1$ we repeat the proof of Claim \ref{claim8}. This proves Claim \ref{claim9} unless $j=k-1$, that is, $G_2$ is the triangle $v_1v_kv_{k-1}v_1$. 
So assume that this is the case.

Then we color $v_{k-1}$, and we apply the induction hypothesis to $G-v_k$. 
If $v_{k-1}$ has precisely three available colors, then both $n$ and $r$ decreases, so Claim \ref{claim9} follows.
If $v_{k-1}$ has at least four available colors, then only $n$ decreases, but there are at least two choices for the color of $v_{k-1}$ unless $G$ is the exceptional case at the end of Theorem \ref{thm:main}. 
So we need only consider the case where $G-v_k$ is the exceptional case at the end of Theorem \ref{thm:main}, namely that $G$ has a vertex with precisely four available colors joined to $v_{k-1}, v_1, v_2$.
Then $k=5$, and $n=2$. As $v_3$ has at least four available colors, $G$ has at least two $(\Z_5, \phi)$-colorings. This proves Claim \ref{claim9}.
\end{proof}

\begin{claim} \label{claim10}
Each non-precolored vertex on $C$ has precisely three available colors.
\end{claim}

\begin{proof}
Suppose for contradiction that Claim \ref{claim10} is false. 
Select a set $S$ of four available colors in $L_{v_i}$ for some vertex $v_i$ on $C$.
Let $S'$ be one of the four 3-element subsets of $S$. 
Now replace $F_{v_i}$ by $\Z_5 \setminus S'$. 
By the maximality of $r$, the new $G$ has at least $2^{n/9-(r+1)/3}$ distinct $(\Z_5,\phi)$-colorings.
As $S'$ can be chosen in four ways, this results in $4 \cdot 2^{n/9-(r+1)/3}$ $(\Z_5,\phi)$-colorings and each of these is counted three times.
Thus we get at least $4 \cdot 2^{n/9-(r+1)/3}/3$ distinct $(\Z_5,\phi)$-colorings, a contradiction which proves Claim \ref{claim10}.
Note that $G$ with its new lists of available colors cannot be a generalized multi-wheel because $n>3r$, as noted earlier.
\end{proof}

\begin{claim} \label{claim11}
$v_k$ is precolored.
\end{claim}

\begin{proof}
Suppose for contradiction that Claim \ref{claim11} is false.
The coloring of $v_1, v_2$ can be extended to $G$.
We give $v_k$ the color in that coloring.
This decreases each of $n,r$ by 1 and hence we obtain a contradiction to the minimality of $n$.
Note that, by Claim \ref{claim10}, the new $G$ cannot have a vertex with precisely four available colors joined to the three colored vertices.
\end{proof}

\begin{claim} \label{claim12}
If $u$ is a vertex in $int(C)$ joined to $v_i,v_j$, where $2\leq i < j \leq k$, then $v$ is also joined to each of $v_{i+1}, v_{i+2}, \ldots, v_{j-1}$.
\end{claim}

\begin{proof}
Suppose for contradiction that there exist $i',j'$ such that $i \leq i' \leq j' - 2 \leq j - 2$ and $u$ is joined to $v_{i'}, v_{j'}$, but not joined to any of $v_{i'+1}, v_{i'+2}, \ldots, v_{j'-1}$.
Let $C'$ be the cycle $u v_{i'} v_{i'+1} \cdots v_{j'} u$, and let $C''$ be the cycle $u v_{j'} v_{j'+1} \cdots v_k v_1 v_2 \cdots v_{i'} u$. 
We apply the induction hypothesis, first to the graph $G'' := C'' \cup int(C'')$ and then to the graph $G' := C' \cup int(C')$.  
This proves Claim \ref{claim12} unless $G'$ is a generalized multi-wheel.
If $G'$ is a generalized multi-wheel, then it is necessarily a multi-wheel, and then, by Lemma \ref{lemma1}, there exists $\alpha \in \Z_5$ such that all colorings of $v_{i'}, u, v_{j'}$ which cannot be extended to $G'$ satisfy $c(v_{j'}) - c(v_{i'}) = \alpha$.
So before we apply the induction hypothesis to $G''$ we add the edge $v_{i'}v_{j'}$ and we let $\phi(v_{i'}v_{j'}) = \alpha$. Apply the induction hypothesis to this graph and then to $G'$.
If $n'$ (respectively $r'$) is the number of non-precolored vertices (respectively non-precolored vertices with precisely three available colors) of $G''$, then it is easy to see that $n'/9 - r'/3 \geq n/9 - r/3$.
This contradiction proves Claim \ref{claim12}.
\end{proof}

Claim \ref{claim12} implies that $G$ does not contain an inserted wheel.

We may assume that 

\begin{claim} \label{claim13}
$k > 4$.
\end{claim}

\begin{proof}
For, if $k=3$, then we delete the edge $v_2v_3$. And if $k=4$, then we color $v_3$ and delete it and use induction after having modified the available lists of the neighbors of $v_3$ accordingly.
\end{proof}

We now split the proof up into the following two cases.

\begin{case}
$G$ does not contain a path $v_2u_1u_2 \cdots u_q v_k$ with the properties that
\begin{enumerate}[label=(\roman*)]
\item
each of $u_1, u_2, \ldots, u_q$ is a vertex in $int(C)$ joined to at least two vertices of $v_3, v_4, \ldots,$ $v_{k-1}$, and
\item
the cycle $v_1 v_2 u_1 u_2 \cdots u_q v_k v_1$ and its interior form a generalized multi-wheel.
\end{enumerate}
\end{case}

\begin{case}
$G$ contains a path $v_2u_1u_2 \cdots u_q v_k$ with the above-mentioned properties (i) and (ii).
\end{case}

Note, that (ii) is equivalent to the following statement: the cycle $v_1 v_2 u_1 u_2 \cdots u_q v_k v_1$ and its interior \textbf{contain} a generalized multi-wheel whose principal path is $v_kv_1v_2$ and all vertices on its outer cycle are on $v_1 v_2 u_1 u_2 \cdots u_q v_k v_1$.
This follows from Claim \ref{claim8} and the fact that if such a subgraph exists, and there is an edge $u_su_t$ for some $s < t$, then we may choose the path $v_2u_1 \cdots u_su_t \cdots u_q v_k$ in Case 2 instead of $v_2u_1u_2 \cdots u_q v_k$.

We first do Case 1. We shall prove that the number of $(\Z_5, \phi)$-colorings is not just at least $2^{n/9-r/3}$ as required in Theorem \ref{thm:main}, but at least $2^{(n+1)/9-r/3}$. 
This will be important in Case 2 which we shall reduce to Case 1 by deleting an appropriate vertex.

Let $R$ be the set of vertices in $int(C)$ which are joined to at least two vertices of the path $C - v_k - v_1 - v_2$.
By Claim \ref{claim12}, the union of the path $C - v_k - v_1 - v_2$ and $R$ and the edges from $R$ to $C$ form a broken wheel string which we will call $W$.

\vspace*{2mm}

\textbf{Subcase 1.1.}
No two consecutive blocks in $W$ are triangles.

We use Lemma \ref{lemma5} to color all the principal neighbors of the major vertices in $W$ in such a way that, regardless of how the major vertices in $W$ are colored, the coloring can be extended to $W$.
This means that we can apply induction to $G' = G - v_3 - v_4 - \cdots - v_{k-1}$.
Any $(\Z_5, \phi)$-coloring of $G'$ can be extended to $G$. 
By the induction hypothesis, the number of $(\Z_5, \phi)$-colorings of $G'$ is at least $2^{n'/9-r'/3}$ where $n' = n-k+3 = n-r$ and $r' = |R|$.
The assumption of Subcase 1.1 implies that $r' \leq (2r-1)/3$. 
Hence the number of $(\Z_5, \phi)$-colorings of $G'$ is at least $2^{(n+1)/9-r/3}$.

\vspace*{2mm}

\textbf{Subcase 1.2.}
Two consecutive blocks in $W$ are triangles.

Let $w_1,w_2$ be two vertices in $R$ each joined to precisely two consecutive blocks of $C$. 
That is, there is a natural number $i$ such that $W$ contains the blocks $w_1 v_{i-1} v_i w_1$ and $w_2 v_i v_{i+1} w_2$.
We now color successively $v_3$ and the cutvertices of $W$ with increasing indices until we color $v_{i-1}$. 
Whenever we color a cutvertex, we do it such that the corresponding block of $W$ can be colored regardless of how we color the major vertex.
This is possible by Lemma \ref{lemma4}.
There are even two possibilities for coloring such a cutvertex of $W$ whenever the preceding cutvertex is a neighbor of the cutvertex that is being colored.
Then we color successively $v_{k-1}$ and the cutvertices of $W$ with decreasing indices until we color $v_{i+1}$.
Again, there are even two possibilities for coloring such a cutvertex of $W$ whenever the preceding cutvertex is a neighbor of the cutvertex that is being colored. 
(Also there are two possibilities for coloring each of $v_3,v_{k-1}$.)
Finally we color $v_i$ and apply the induction hypothesis to $G - v_3 - v_4 - \cdots - v_{k-1}$.
Let $r'$ be the number of vertices of $R$ and let $n'$ be the number of uncolored vertices of $G - v_3 - v_4 - \cdots - v_{k-1}$. Then $n' = n-k+3= n-r$.

The number of colorings of the vertices of $W$ in the path $v_2v_3 \cdots v_{k-1}$ is at least $2^t$, where $t$ is the number of blocks of $W$ which are triangles.

For each of these there are at least $2^{n'/9-r'/3}$ $(\Z_5,\phi)$-colorings of $G - v_3 - v_4 - \cdots - v_{k-1}$, by the induction hypothesis.
Let $s$ be the number of blocks of $W$ which are not triangles. Then $r' = s + t$ and $r \geq 2s+t+1$. 
So the total number of colorings of $G$ is at least $2^{n'/9 - r'/3+t}$ which is greater than $2^{(n+1)/9 - r/3}$. 
This completes the proof in Case 1.

\vspace*{3mm}

We now do Case 2. Let $m$ be the smallest number such that $u_q$ is joined to $v_m$.
By Claim \ref{claim12}, $u_q$ is joined to $v_m, v_{m+1}, \ldots, v_k$ (and possibly also to $v_1$).
Again, we split up into two cases.

\vspace*{3mm}

\textbf{Subcase 2.1.}
$v_1$ is joined to $u_q$.

We select two colors $\alpha, \beta$ in $L_{v_{k-1}}$ distinct from $\tau_{v_k}(v_{k-1})$.
We add the colors $\tau_{v_{k-1}}(\alpha,u_q)$, $\tau_{v_{k-1}}(\beta,u_q)$ to $F_{u_q}$ and we delete the vertex $v_{k-1}$ from $G$.
Then we color $u_q$ and delete also $v_k$.
By the induction hypothesis, if the resulting graph $G'$ has at least one $(\Z_5, \phi)$-coloring, then it has at least $2^{(n-2)/9 - (r-1)/3}$ $(\Z_5, \phi)$-colorings. 
Each such coloring can be extended to $v_{k-1}$ and the proof is complete.
So assume that $G'$ has no $(\Z_5, \phi)$-coloring. 
By Theorem \ref{thm:3}, $G'$ contains a generalized multi-wheel. 
Clearly, $q \leq 2$. Furthermore, $G'$ has no chords $v_1v_i$ for $3 \leq i \leq k-2$. 
Hence either $q=1$ in which case $G$ is a wheel by Claims \ref{claim8} and \ref{claim12}, or else $q=2$ in which case $G'-v_{m+1}-v_{m+2}- \cdots - v_{k-2}$ is a wheel. 
But then $n - r \leq 2$ and there is nothing to prove.

\vspace*{2mm}

\textbf{Subcase 2.2.}
$v_1$ is not joined to $u_q$. Now $G$ has a vertex $w$ joined to $v_k, u_q, u_{l}$ for some $l < q$ by the definition of a generalized multi-wheel.
By Claim \ref{claim12}, $w$ is not joined to $v_2$.

If $m < k-2$, then we select two colors $\alpha, \beta$ in $L_{v_{k-1}}$ distinct from $\tau_{v_k}(v_{k-1})$. We add the colors $\tau_{v_{k-1}}(\alpha,u_q), \tau_{v_{k-1}}(\beta,u_q)$ to $F_{u_q}$ and we delete the vertices $v_{k-1}, v_{k-2}, \ldots, v_{m+1}$ from $G$. 
Then we use the induction hypothesis to obtain a contradiction because the resulting graph has a smaller $r$. So assume that $m=k-2$.

If $q=1$, then $w$ is joined to $v_2$, hence by Claim \ref{claim12} both of $w$ and $u_1$ are joined to all of $v_3, v_4, \ldots, v_{k-1}$ which is impossible. So assume that $q>1$.

If $u_{q-1}$ is joined to $v_{k-2}$, then we select two colors $\alpha, \beta$ in $L_{v_{k-1}}$ distinct from $\tau_{v_k}(v_{k-1})$. We add the colors $\tau_{v_{k-1}}(\alpha,u_q), \tau_{v_{k-1}}(\beta,u_q)$ to $F_{u_q}$ and we delete the vertex $v_{k-1}$ from $G$. 
The resulting graph $G'$ satisfies the assumption in Case 1. We explain why: If $G'$ contains a path $v_2w_1w_2 \cdots w_{q'}v_k$ such that each of $w_1,\ldots, w_{q'}$ is a vertex in the interior of $v_1v_2\cdots v_{k-2} u_qv_kv_1$, then $w_{q'}$ cannot be joined to any of $v_3, \ldots, v_{k-2}$, hence it is not joined to at least two vertices of $v_3, v_4, \ldots, v_{k-2}, u_q$.
Thus the conclusion follows by repeating the proof in Case 1. 
The reason we can repeat the proof in Case 1 is that $G'$ satisfies the analogue of Claim \ref{claim12} when $u_{q-1}$ is joined to $v_{k-2}$.
Therefore we may assume that $u_{q-1}$ is not joined to $v_{k-2}$. 

Let $i$ be the smallest number such that $u_{q-1}$ is joined to $v_i$, and let $j$ be the largest number $u_{q-1}$ is joined to $v_j$. Then $j < k-2$.
We select two colors $\alpha, \beta$ in $L_{v_{k-1}}$ distinct from $\tau_{v_k}(v_{k-1})$.
We add the colors $\tau_{v_{k-1}}(\alpha,u_q), \tau_{v_{k-1}}(\beta,u_q)$ to $F_{u_q}$ and we delete the vertex $v_{k-1}$.
The path $v_j u_{q-1} u_q$ divides the resulting graph into two graphs $G_1, G_2$, where $G_1$ contains $v_1$. 
Assume first that $G_2$ is a generalized multi-wheel.
If $G_1$ contains a generalized multi-wheel then $r \geq n/3$, so there is nothing to prove. 
If not, then we obtain a contradiction by applying the induction hypothesis to $G_1$. Let $n'$ denote the number of non-precolored vertices in $G_1$, and let $r'$ denote the number of vertices in $G_1$ with precisely three available colors. As $G_2$ is a generalized multi-wheel without chords, it is a multi-wheel and we have $n'/9-r'/3 \geq n/9-r/3$.
We use the fact that, by Lemma \ref{lemma1}, there exists $\alpha \in \Z_5$ such that all colorings of $v_j, u_{q-1}, u_q$ which cannot be extended to $G_2$ satisfy $c(u_q) - c(v_{j}) = \alpha$.
So before we apply the induction hypothesis to $G_1$ we add the edge $v_ju_q$ (directed towards $u_q$) and we let $\phi(v_ju_q) = \alpha$.
In this case the number of $(\Z_5, \phi)$-colorings of $G_1$ is greater than or equal to $2^{n/9-r/3}$, and any such coloring can be extended to $G_2$.

On the other hand, if $G_2$ is not a generalized multi-wheel, then we obtain a contradiction by applying the induction hypothesis first to $G_1$ and then to $G_2$.
We lose a multiplicative factor $2^{1/9}$ because of the deleted vertex $v_{k-1}$.
We make up for that before we apply induction to $G_1$ since we can delete one of the available colors of $u_{q-1}$ in at least five different ways.
In this way we gain a multiplicative factor $5/4$, and now the proof is complete, because $5/4 > 2^{1/9}$.
\end{proof}

\begin{cor}
Every planar simple graph with $n$ vertices has at least $2^{n/9}$ $\Z_5$-colorings. 
\end{cor}

\end{document}